\documentclass[12pt]{amsart} 

\usepackage{verbatim, amssymb, enumerate}
\usepackage{hyperref}

\renewcommand \a{\alpha}
\renewcommand \b{\beta}

\newcommand \la{\lambda}
\newcommand \ve{\varepsilon}
\newcommand \id{\mathrm{id}}
\newcommand \br{\mathbb{R}}
\newcommand \bc{\mathbb{C}}

\newcommand \Ker{\operatorname{Ker}}
\newcommand \Der{\operatorname{Der}}
\newcommand \End{\operatorname{End}}
\newcommand \Ric{\operatorname{Ric}}

\renewcommand \Re{\operatorname{Re}}

\newcommand \Span{\operatorname{Span}}
\newcommand \Tr{\operatorname{Tr}}

\newcommand \cB{\mathcal{B}}

\newcommand \ig{\mathfrak{i}}
\newcommand\ag{\mathfrak a}
\newcommand\g{\mathfrak g}
\newcommand\h{\mathfrak h}
\newcommand\z{\mathfrak z}
\newcommand\m{\mathfrak m}

\newcommand \gl{\mathfrak{gl}}

\newcommand \s{\mathfrak{s}}
\newcommand \n{\mathfrak{n}}

\newcommand \ad{\operatorname{ad}}

\newcommand \diag{\operatorname{diag}}

\DeclareMathOperator{\ric}{Ric}

\newcommand \<{\langle}
\renewcommand \>{\rangle}
\newcommand \ip{\<\cdot,\cdot\>}

\newcommand \GL{\mathrm{GL}}

\makeatletter
\newtheorem*{rep@theorem}{\rep@title}
\newcommand{\newreptheorem}[2]{%
\newenvironment{rep#1}[1]{%
\def\rep@title{#2 \ref{##1}}%
\begin{rep@theorem}}%
{\end{rep@theorem}}}
\makeatother

\theoremstyle{plane}
\newtheorem{theorem}{Theorem}
\newtheorem*{theorem*}{Theorem}
\newtheorem*{corollary*}{Corollary}
\newtheorem*{conj*}{Conjecture}
\newtheorem{lemma}{Lemma}
\newtheorem{proposition}{Proposition}
\newtheorem*{prop*}{Proposition}
\newreptheorem{theorem}{Theorem}

\theoremstyle{definition}

\newtheorem*{definition*}{Definition}

\theoremstyle{remark}
\newtheorem{remark}{Remark}

\begin{document}

\title{Solvable Lie groups of negative Ricci curvature}

\author{Y.~Nikolayevsky}
\address{Y.~Nikolayevsky, Department of Mathematics and Statistics, La Trobe University, Melbourne, Australia 3086}
\email{y.nikolayevsky@latrobe.edu.au}

\author{Yu.G.~Nikonorov}
\address{Yu.G.~Nikonorov, South Mathematical Institute of VSC RAS, 22 Markus st, \linebreak Vladikavkaz, Russia 362027}
\email{nikonorov2006@mail.ru}



\thanks{The first author is partially supported by ARC Discovery Grant DP130103485. The second author is supported in part by the
State Maintenance Program for the Leading Scientific Schools of the Russian Federation (grant NSh-921.2012.1) and by Federal Target
Grant ``Scientific and educational personnel of innovative Russia'' for 2009-2013 (agreement no. 8206, application no. 2012-1.1-12-000-1003-014).
}


\begin{abstract}
We consider the question of whether a given solvable Lie group admits a left-invariant metric of strictly negative Ricci curvature.
We give necessary and sufficient conditions of the existence of such a metric for the Lie groups the nilradical of whose Lie algebra
is either abelian or Heisenberg or standard filiform, and discuss some open questions.
\end{abstract}

\maketitle

\section{Introduction}
\label{s:intro}

The fundamental question of Riemannian geometry is whether (and when) a given manifold admits a Riemannian metric with a particular sign of the curvature. Naturally, in the context of homogeneous geometry, the same question is being asked for left-invariant metrics. In that case the curvature is entirely expressed in terms of the algebraic structure of the given homogeneous space and one expects the answer to be stated in both topological and algebraic terms.

The conditions on the sign of the \emph{sectional curvature} $K_\sigma$ are quite restrictive and the homogeneous spaces whose sectional curvature has a particular sign are mainly understood. For $K_\sigma > 0$ the question was settled in \cite{Wal, BB1} who showed that a homogeneous space admits a left-invariant metric with $K_\sigma > 0$ if and only if it is diffeomorphic to either CROSS or to a space from a short finite list (so-called Wallach and Allof-Wallach spaces). By \cite{Ale, Hei}, a homogeneous space of negative sectional curvature is isometric to a solvmanifold the nilradical $\n$ of whose Lie algebra $\g$ has codimension one and there exists $Y \in \g \setminus \n$ such that all the eigenvalues of the restriction of $\ad_Y$ to $\n$ have positive real part. Flat homogeneous spaces were completely described in \cite{Ale, BB2}: every such space is isometric to a solvmanifold the nilradical $\n$ of whose Lie algebra $\g$ is abelian and the restrictions of all $\ad_Y, \; Y \in \g \setminus \n$, to $\n$ are skew-symmetric.

For the Ricci curvature, the positive case was settled down by Milnor \cite{Mil} (for Lie groups), and by Berestovskii \cite{Ber} (in the general case), who proved that a homogeneous space admits a left-invariant metric with $\Ric > 0$ if and only if it is compact and has a finite fundamental group.
By a beautiful result of \cite{AK}, any Ricci-flat homogeneous space is flat.

Much less is known, however, about Riemannian homogeneous spaces of negative Ricci curvature. By the following theorem, no unimodular solvable Lie group (in particular, no nilpotent group) admits a left-invariant metric with $\Ric < 0$:
\begin{theorem}[\cite{DM}]\label{t:dm}
Any left-invariant metric with $\Ric \le 0$ on a solvable unimodular Lie group is Ricci-flat.
\end{theorem}
By \cite{AK}, any such metric is flat. Further in \cite{DLM} it was proved that a unimodular Lie group which admits a left-invariant metric with $\Ric < 0$ is noncompact and semisimple. Examples of such metrics were constructed on
$\mathrm{SL}(n,\mathbb{R}), \; n \ge 3$ \cite{LDM} and on some complex simple Lie groups \cite{DLM}. To the best of our knowledge, however, the general (nonunimodular) case has not been studied in the literature, although an important subclass of left-invariant metrics with negative Ricci curvature, the Einstein metrics of negative scalar curvature, has been extensively studied in the past decades by many authors, including the present authors.

In this paper we ask the following question: \emph{Which solvable Lie groups admit a left-invariant metric of negative Ricci curvaturee?}

The Ricci curvature of a left-invariant metric on a Lie group $G$ can be entirely computed from the algebraic data: the structure of the Lie algebra $\g$ of $G$ and the inner product $\ip$ on $\g$ (see Section~\ref{ss:ric} for details). With a slight abuse of terminology, we will speak of the Ricci curvature of the metric Lie algebra $(\g, \ip)$.

After some preliminaries (Section~\ref{s:pre}), we prove in Section~\ref{s:abelian} the following theorem, which gives some necessary conditions and some sufficient conditions for the existence of an inner product of negative Ricci curvature on a solvable Lie algebra.

\begin{theorem}\label{t:neg}
Suppose $\g$ is a solvable Lie algebra. Let $\n$ be the nilradical of $\g$ and $\z$ be the centre of $\n$. Then
\begin{enumerate}[{\rm (1)}]
  \item \label{it:neg1}
  If $\g$ admits an inner product of negative Ricci curvature, then there exists $Y \in \g$ such that $\Tr \ad_Y > 0$ and all the eigenvalues of the restriction of the operator $\ad_Y$ to $\z$ have a positive real part;

  \item \label{it:neg2}
  If there exists $Y \in \g$ such that all the eigenvalues of the restriction of $\ad_Y$ to $\n$ have positive real part, then $\g$ admits an inner product of negative Ricci curvature.
\end{enumerate}
\end{theorem}

The following is an immediate consequence of Theorem~\ref{t:neg}.

\begin{corollary*}
A solvable Lie algebra $\g$ with an abelian nilradical $\n$ admits an inner product of negative Ricci curvature if and only if there exists $Y \in \g$
such that all the eigenvalues of the restriction of $\ad_Y$ to $\n$ have positive real part.
\end{corollary*}

Note that the latter property of $\g$ in the Corollary is equivalent to the following one:
There exists $Y' \in \g$ such that the restriction of $\ad_{Y'}$ to $\n$ is {\it a stable linear operator}, i.~e. all its eigenvalues have negative real parts
(see e.~g. \cite{Gant} or \cite{LiW} for a discussion on stable linear operators and stable matrices).

We further develop the approach taken in the Corollary (obtaining the condition for $\Ric < 0$ for classes of solvable algebras with a given nilradical) for two important classes of nonabelian nilpotent algebras serving as nilradicals: the Heisenberg algebra and the standard filiform algebra.

Recall that the \emph{Heisenberg Lie algebra} $\h_{2p+1}$ of dimension $l=2p+1, \; p \ge 1$, is defined by the relations $[X_i, X_{p+i}]=Z$ for $i=1, \dots, p$, relative to a basis  $\{X_1, \dots, X_{2p}, Z\}$, where $[X_i,X_j]=0$ if $|i-j| \ne p$, and $Z$ spans the centre $\z$ of $\h_{2p+1}$. Let $\g$ be a solvable Lie algebra with the nilradical $\n=\h_{2p+1}$. For any $X \in \g$, the vector $Z$ is an eigenvector of the restriction of $\ad_X$ to $\n$, so that $[X,Z]=\la(X)Z$ for a one-form $\la$ on $\g$, and moreover, $(\ad_X)_{|\n}$ descends to a well-defined linear map $\Phi(X) \in \End(\n/\z)$. Let $d_i(X) \in \bc, \; i=1, \dots, 2p$, be the eigenvalues of $\Phi(X)$, each listed with its algebraic multiplicity. In Section~\ref{s:heis} we prove the following theorem.

\begin{theorem}\label{t:heis}
A solvable Lie algebra $\g$ with the Heisenberg nilradical admits an inner product of negative Ricci curvature if and only if there exists $Y \in \g$ such that in the above notation,
\begin{equation*}
    \la(Y) + \sum\nolimits_{i: \Re \, d_i(Y) < 0} \Re \, d_i(Y) > 0.
\end{equation*}
\end{theorem}

Further on, in Section~\ref{s:fili} we consider solvable Lie algebras whose nilradical is \emph{filiform} (has the maximal possible degree of nilpotency) \cite{Ver}. More specifically, we require the nilradical to be \emph{standard filiform Lie algebra}, which is defined as the $l$-dimensional Lie algebra $L_l$ having a basis $X_1, \dots, X_l$ such that $[X_1, X_i]=X_{i+1}, \; i=2, \dots, l-1$, $[X_1, X_l]=0$, and $[X_i,X_j]=0$ when $i,j \ge 2$. Note that any filiform Lie algebra admits a basis for which the former relations are satisfied (but in general, not the latter ones).

Let $\g$ be a solvable Lie algebra with the nilradical $L_l$. We can assume that $l \ge 4$ (as $L_2$ is abelian and $L_3$ is the Heisenberg algebra). The algebra $L_l$ has a (unique) codimension one abelian ideal $\ig=\Span(X_2, \dots, X_l)$ and the one-dimensional centre $\br X_l$. Both of them are characteristic ideals of $L_l$ (they are invariant under the action of any derivation on $L_l$; see Section~\ref{s:fili}). Define the one-forms $\la$ and $\iota$ on $\g$ as follows: for $Y \in \g, \; [Y,X_l]=\la(Y) X_l$ and $\iota(Y)=\Tr((\ad_Y)_{|\ig})$. We have the following theorem.

\begin{theorem}\label{t:fili}
Let $\g$ be a solvable Lie algebra with the nilradical $\n=L_l, \; l \ge 4$. The algebra $\g$ admits an inner product of negative Ricci curvature if and only if there exists $Y \in \g$ such that $\la(Y) > 0$ and $\iota(Y) > 0$.
\end{theorem}

In the last section we collect some open questions and conjectures. The main of them, motivated by the above theorems is the following
(in a slightly vague formulation). Is it true that a solvable Lie algebra $\g$ admits an inner product with $\Ric < 0$ if and only if there exists a vector $Y \in \g$ such that the real parts of the eigenvalues of the restriction of $\ad_Y$ to the nilradical $\n$ of $\g$ satisfy certain linear inequalities which are determined by the structure of $\n$? Speculating a little further we might suggest that such inequalities represent the fact that $\Re \, (\ad_Y)_{|\n}$ belongs to a certain open convex hull (see the details in Section~\ref{s:q}).

\section{Preliminaries}
\label{s:pre}

\subsection{The Ricci operator}
\label{ss:ric}

Let $G$ be a Lie group with a left-invariant metric $Q$
obtained by the left translations from an inner product $\ip$ on the Lie algebra $\g$ of $G$.
Let $B$ be the Killing form of $\g$, and let $H \in \g$ be the \emph{mean curvature vector} defined by
$\<H, X\> = \Tr \ad_X$.

The Ricci curvature $\mathrm{ric}$ of the metric Lie group $(G,Q)$ at the identity is given by
\begin{equation*}
    \mathrm{ric}(X)=-\<[H,X],X\>-\frac12 B(X,X)-\frac12 \sum\nolimits_i \|[X,E_i]\|^2 +\frac14 \sum\nolimits_{i,j} \<[E_i,E_j],X\>^2,
\end{equation*}
for $X \in \g$, where $\{E_i\}$ is an orthonormal basis for $(\g, \ip)$ (e.g. \cite{Ale}). Equivalently, one can define the Ricci operator $\ric$ of the metric Lie algebra $(\g, \ip)$, the symmetric operator associated to $\mathrm{ric}$, by

\begin{equation} \label{eq:riccAl}
\ric = -\frac{1}{2} \sum\nolimits_i \ad_{E_i}^{t} \ad_{E_i} + \frac{1}{4} \sum\nolimits_i \ad_{E_i}\ad_{E_i}^{t}-\frac{1}{2}B -(\ad_H)^{s},
\end{equation}
where $A^t$ is the operator adjoint to $A$ and $(\ad_H)^{s}=\frac12(\ad_H+\ad_H^t)$ is the symmetric part of $\ad_H$.

If $(\n, \< \cdot, \cdot \>)$ is a nilpotent metric Lie algebra, then $H = 0$ and $B = 0$, and we get

\begin{equation}\label{eq:riccinilexplicit}
\begin{gathered}
\<\ric^{\n} X, Y \> = \frac14 \sum\nolimits_{i,j} \<X, [E_i, E_j]\> \<Y, [E_i, E_j]\> -\frac12 \sum\nolimits_{i,j} \<[X, E_i], E_j\> \<[Y, E_i], E_j]\>.\\
\ric^{\n} = -\frac{1}{2} \sum\nolimits_i \ad_{E_i}^{t} \ad_{E_i} + \frac{1}{4} \sum\nolimits_i \ad_{E_i}\ad_{E_i}^{t}.
\end{gathered}
\end{equation}

We will use a more explicit form of \eqref{eq:riccAl} in the case when $\g$ is solvable. Denote $\n$ the nilradical of $\g$, the maximal nilpotent ideal of $\g$. Clearly, $[\g,\g]\subset \n$, but in general, $[\g,\g] \neq \n$. It is known that $\dim(\n)\geq\frac12\big(\dim(\g)+\dim(\z(\g)\big)$,
where $\z(\g)$ is the center of $\g$ \cite[Theorem~5.2]{VGO}. Denote $\ag=\n^\perp$ and put $l =\dim\n$ and $m =\dim\ag$ (we call $m$ the \emph{rank} of $\g$), with $l+m=n=\dim\g$. Choose orthonormal bases $\{e_i\}$ for $\n$, and $\{f_k\}$ for $\ag$ in such a way that
\begin{equation*}
t:=\Tr(\ad(f_1))\geq 0,  \quad  \Tr(\ad(f_j))=0,  \quad  2 \leq j \leq m.
\end{equation*}
It is easy to see that for a non-unimodular Lie algebra $\g$ we have $f_1=\|H\|^{-1}H, \; t= \Tr(\ad(f_1))=\|H\|>0$. If $\g$ is unimodular, we can choose $f_1 \in \ag$ arbitrarily (and $t=0$).

Relative to the basis $\{e_1,...,e_l, f_1,...,f_m\}$, the matrices of the operators $\ad_{f_j}$ and $\ad_{e_i}$ have the form
\begin{equation}\label{eq:adad}
\ad_{f_j} = \left(
\begin{array}{cc}
 A_j  & B_j \\
 0  & 0 \\
 \end{array}
 \right), \quad
\ad_{e_i} = \left(
\begin{array}{cc}
 D_i  & C_i\\
 0 & 0 \\
 \end{array}
 \right),
\end{equation}
for some $(l \times l)$-matrices $A_j$, $D_i$  and $(l \times m)$-matrices $B_j$, $C_i$, and the matrix of the Ricci operator of the solvable metric Lie algebra $(\g, \ip)$ has the form (see the proof of \cite[Theorem~3]{NikNik})
\begin{equation}
\label{nnteor}
\Ric = \left( {{\begin{array}{*{20}c}
 R_1 \hfill & R_2 \hfill\\
 R_2^t \hfill & R_3 \hfill\\
 \end{array} }} \right),
\end{equation}
where
\begin{align} \label{eq:r1}
R_1 &= \Ric^\mathfrak{n} + \frac{1}{2} \sum\nolimits_{j=1}^{m}[A_j,A_j^t] + \frac{1}{4} \sum\nolimits_{j=1}^{m}B_jB_j^t - t A_1^s, \\
R_2 &= -\frac{1}{2} \Big( \sum\nolimits_{i=1}^{l}D_i^tC_i +\sum\nolimits_{j=1}^{m} A_j^tB_j +t B_1 \Big), \label{eq:r2}\\
R_3 &= -\frac{1}{2}\sum\nolimits_{j=1}^{m}B_j^tB_j-L \label{eq:r3},
\end{align}
where $L$ is an $(m \times m)$-matrix with the entries $L_{pq}= \Tr(A_p^sA_q^s)$, $A_j^s = \frac{1}{2} (A_j^t+A_j)$ is the \emph{symmetric part} of $A_j$, $t=\Tr(A_1)=\Tr(A_1^s) \geq 0$ and $\Ric^\n$ is the matrix of the Ricci operator of the metric nilpotent Lie algebra $(\n, \ip_{\n})$ relative to the basis $\{e_1,....,e_l\}$ which from \eqref{eq:riccinilexplicit} is easily seen to be given by

\begin{equation}
\label{eq:riccAln}
\Ric^\mathfrak{n} = -\frac{1}{2} \sum\nolimits_{i=1}^{l} D_i^{t}D_i +\frac{1}{4} \sum\nolimits_{i=1}^{l}D_iD_i^{t}.
\end{equation}

\subsection{Orbit closure}
\label{ss:clo}

Let $\g$ be a Lie algebra with the underlying linear space $\br^n$. We denote $\mu$ the Lie bracket of $\g$, so that $\mu: \Lambda^2\br^n \to \br^n$ is the defined by $\mu(X,Y)=[X,Y]$, for $X,Y \in \br^n$. The map $\mu$ is an element of the space $\mathcal{L} \subset \Lambda^2(\br^{n*}) \otimes \br^n$ of Lie brackets on $\br^n$ (skew-symmetric bilinear maps satisfying the Jacobi identity). The space $\mathcal{L}$ is acted upon by the group $\GL(n)$ as follows (``change of basis"): for $T \in \GL(n)$ and $\mu \in \mathcal{L}$ we define $T.\mu \in \mathcal{L}$ by $(T.\mu)(X,Y)=T^{-1}\mu(TX,TY)$. It is clear that the Lie algebra defined by the bracket $T.\mu$ on $\br^n$ is isomorphic to the one defined by the bracket $\mu$. As $\mathcal{L}$ is defined by polynomial equations, any element of the closure of the orbit $\GL(n).\mu$ of $\mu$ (in the usual topology of $\Lambda^2(\br^{n*}) \otimes \br^n$) is again a Lie bracket, but the corresponding Lie algebra $\bar{\g}(\br^n, \nu)$ may not be isomorphic to $(\br^n, \mu)$. We say that $\bar{\g}$ is a \emph{degeneration} of $\g$ ($\bar{\g}$ is usually ``more abelian" than $\g$, see e.g. \cite{Bur, NP}).

The following proposition is elementary, but useful.

\begin{proposition}\label{p:clo}
Suppose $\mu$ and $\nu$ are Lie brackets on $\br^n$ such that $\nu$ belongs to the closure of the $\GL(n)$ orbit of $\mu$. If the Lie algebra $(\br^n, \nu)$ admits an inner product of negative Ricci curvature, then so does the Lie algebra $(\br^n, \mu)$.
\end{proposition}
\begin{proof}
Let $\ip$ be such an inner product on $(\br^n, \nu)$. As the Ricci tensor depends continuously on the structural constants of the Lie algebra relative to a fixed basis, there is a $T \in \GL(n)$ such that the metric Lie algebra $(\br^n, T.\mu, \ip)$ has negative Ricci curvature. But the latter Lie algebra is isomorphic to the Lie algebra $(\br^n, \mu)$.
\end{proof}

\subsection{Technical lemmas}
\label{ss:tech}

We will need the following two facts.

The first one is a ``real" modification of the Lie Theorem (this should be well-known; we supply a short proof for the sake of completeness).
\begin{lemma}\label{l:reallie}
Let $\s$ be a real solvable Lie algebra and $V$ be a real $\s$-module. Then there exists a basis $\mathcal{B}$ for $V$ relative to which all the elements $g \in \s$ are block-lower-triangular, with the diagonal blocks of sizes either $1 \times 1$ with the corresponding diagonal entry $\la_j(g)$, or $2 \times 2$, with the corresponding diagonal block $G_i(g)= \begin{pmatrix} \a_i(g) & \b_i(g) \\ -\b_i(g) & \a_i(g) \end{pmatrix}$, where $\la_j, \a_i, \b_i$ are linear forms on $\s$ such that $\b_i \ne 0$.
\end{lemma}
\begin{proof}
The proof essentially mimics that for the classical Lie Theorem. All the elements $g \in \s$ have a common eigenvector $X \in V^\bc$. If $X$ is real, then $gX=\la(g)X$ for some linear form $\la$ on $V$, and $V/(\br X)$ is again an $\s$-module. Otherwise, $X=X_1+X_2 \mathrm{i}$, where $X_1, X_2 \in V$ are non-collinear, in which case $L=\Span(X_1,X_2)$ is an $\s$-module and the representation $G$ of $\s$ on $L$ relative to the basis $\{X_1,X_2\}$ is given by $G(g)=\begin{pmatrix} \a(g) & \b(g) \\ -\b(g) & \a(g) \end{pmatrix}$, for some linear forms $\a, \b$ on $\s,\; \b \ne 0$. Then $V/L$ is again a $\s$-module and the proof follows by induction.
\end{proof}

We will also need the following ``general position" lemma.

\begin{lemma}\label{l:nonskew}
Let $V$ be a real linear space and $\s \subset \gl(V)$ a solvable subalgebra. Then for an open, dense set of inner products on $V$, no nonzero element of $\s$ is skew-symmetric.
\end{lemma}
\begin{proof}
Let $\cB_0$ be an arbitrary basis for $V$ and let $\ip_0$ be the inner product for which $\cB_0$ is orthonormal. Then, relative to $\cB_0$, the inner products $\ip$ on $V$ are in one-to-one correspondence with symmetric positive definite matrices $Q$, so that $\<X,Y\>=\<QX,Y\>_0$, for all $X,Y \in V$. Choose a basis for $\s$ represented by the matrices $M_1, \dots, M_m$ relative to $\cB_0$. The fact that for the inner product corresponding to a matrix $Q$ the subalgebra $\s$ contains a nonzero skew-symmetric operator is equivalent to the fact that the matrices $QM_i+M_i^tQ$ are linearly dependent, which is a polynomial condition for the entries of $Q$. It follows that it suffices to find at least one $Q$ for which it is violated, which is equivalent to finding a basis for $V$ relative to which no nonzero element of $\s$ is represented by a skew-symmetric matrix. Take a basis $\cB$ constructed in Lemma~\ref{l:reallie} and modify it as follows: for every pair of basis vectors corresponding to a $2 \times 2$ diagonal block $G_i$, multiply one of them by $2$. Then no nonzero matrix from $\s$ is skew-symmetric relative to the resulting basis.
\end{proof}

\section{Abelian nilradical and the proof of Theorem~2}
\label{s:abelian}

Theorem~\ref{t:neg} implies the Corollary which answers our question for solvable Lie algebras with the abelian nilradical. For the convenience of the reader, we reproduce it here (with a slight change of notation).

\begin{reptheorem}{t:neg}
Suppose $\g$ is a solvable Lie algebra. Let $\n$ be the nilradical of $\g$ and $\z$ be the centre of $\n$. Then
\begin{enumerate}[{\rm (1)}]
  \item
  If $\g$ admits an inner product of negative Ricci curvature, then there exists $X \in \g$ such that $\Tr \ad_X > 0$ and all the eigenvalues of the restriction of the operator $\ad_X$ to $\z$ have a positive real part;

  \item
  If there exists $X \in \g$ such that all the eigenvalues of the restriction of $\ad_X$ to $\n$ have positive real part, then $\g$ admits an inner product of negative Ricci curvature.
\end{enumerate}
\end{reptheorem}

\begin{proof}
(1) First of all, note that $\z$ is an ideal of $\g$, so for any $Y \in \g$, the restriction of $\ad_Y$ to $\z$ is well-defined. Moreover, as $\z$ is abelian, all such restrictions commute. Suppose that for an inner product $\ip$ on $\g$, the Ricci curvature is negative.

Take $X=f_1$. Let $\xi \in \bc$ be an eigenvalue of the restriction of $\ad_{f_1}$ to $\z$ and let $V_\xi \subset \z^\bc$ be the corresponding eigenspace. As the operators $(\ad_{f_j})_{|\z^\bc}$ commute, the subspace $V_\xi$ is $\ad_{f_j}$-invariant, for all $j=1, \dots, m$, and moreover, the restrictions of $\ad_{f_j}$ to $V_\xi$ again commute. By the Lie Theorem (or by Frobenius Theorem, \cite{New}), the latter have a common eigenvector, so that there exists $Z \in V_{\xi} \setminus 0$ such that $\ad_{f_j}Z=\xi_j Z$, and $\xi_1=\xi$. Let $Z=Z_1+Z_2 \mathrm{i}, \; Z_1, Z_2 \in \z$. Multiplying $Z$ by a nonzero complex number we can assume that $\|Z_1\|=1$ and that $Z_1 \perp Z_2$ (note that $Z_2$ can be zero). We have $A_jZ_1=\a_j Z_1 - \b_j Z_2, \; A_j Z_2=\a_j Z_2+ \b_j Z_1$, where $\xi_j=\a_j + \b_j \mathrm{i}, \; j=1, \dots, m$. Computing $\<R_1Z_1,Z_1\>+\<R_1Z_2,Z_2\>$ by \eqref{eq:r1} and \eqref{eq:riccAln} we obtain $\<(-\tfrac12 \sum_{i=1}^l D_i^tD_i +\tfrac14 \sum_{i=1}^l D_iD_i^t)Z_k,Z_k\>=\tfrac14 \sum_{i=1}^l \|D_i^tZ_k\|^2 \ge 0$, for $k=1,2$, as $D_iZ_k=0$, so $\<\Ric^\n Z_k,Z_k\> \ge 0$. Moreover, $\<B_jB_j^tZ_k,Z_k\> \ge 0$. Furthermore, for $j=1, \dots, m$ we have $\<([A_j,A_j^t]Z_1,Z_1\>+\<([A_j,A_j^t]Z_2,Z_2\>= \|A_j^tZ_1\|^2-\|A_jZ_1\|^2+\|A_j^tZ_2\|^2-\|A_jZ_2\|^2 = \sum_{i=1}^{l}(\<A_je_i,Z_1\>^2+\<A_je_i,Z_2\>^2) - \|A_jZ_1\|^2 - \|A_jZ_2\|^2$. Specifying the orthonormal basis $\{e_i\}$ in such a way that $e_1=Z_1$ and $e_2=\|Z_2\|^{-1}Z_2$ if $Z_2 \ne 0$ (or arbitrary otherwise) we find that the latter expression equals $\sum_{i=3}^{l}(\<A_je_i,Z_1\>^2+\<A_je_i,Z_2\>^2)$ (or $\sum_{i=2}^{l}\<A_je_i,Z_1\>^2$, respectively). So in the both cases, $\<([A_j,A_j^t]Z_1,Z_1\>+\<([A_j,A_j^t]Z_2,Z_2\> \ge 0$. Finally we have $\<A_1^sZ_1,Z_1\>+\<A_1^sZ_2,Z_2\>= \<A_1Z_1,Z_1\>+\<A_1Z_2,Z_2\>= \a_1(1+\|Z_2\|^2)$, It follows that $\<R_1Z_1,Z_1\>+\<R_1Z_2,Z_2\> \ge \a_1(1+\|Z_2\|^2) \Tr A_1$, which implies $\Re(\xi) \Tr A_1 < 0$.

(2)
By Lemma~\ref{l:reallie}, we can choose a basis $\mathcal{B}'=\{X_i\}_{i=1}^l$ for $\n$ relative to which all the operators $(\ad_Y)_{|\n}$, $Y \in \g$, are block-lower-triangular, with $p$ $(1 \times 1)$-blocks and $q$ $(2 \times 2)$-blocks, where $2q+p=l, \; q, p \ge 0$.

For $a=1, \dots, p+q$, let $l_a \in \{1,2\}$ be the dimension of the $a$-th diagonal block (counting from the top-left to the bottom-right corner), and let $N$ be a diagonal matrix (relative to $\cB'$) defined by $N=\diag(d_1I_{l_1}, d_2I_{l_2}, \dots, d_{p+q}I_{l_{p+q}})$, where $d_a, \; a=1, \dots, p+q$, is a sequence of positive numbers such that $d_a+d_b < d_{\max(a,b)+1}$ (say $d_a=3^a$).

Extend the basis $\mathcal{B}'$ to a basis $\mathcal{B}$ for $\g$ by elements $\{Y_k\}_{k=1}^m$, and for $s \in \br$, define $T_s \in \End(\g)$ by $T_sX=e^{sN}X$ for $X \in \n$ and $T_sY_k=Y_k$. Let $\mu$ be the Lie bracket of $\g$ and let $\nu=\lim_{s \to \infty}T_s.\mu$. By the choice of the $d_a$, it follows that $\nu(X_i,X_j)=0$; moreover, as $d_a > 0$, we have $\nu(Y_k,Y_r)=0$. Again, by the choice of $d_a$, it follows from that for the bracket $\nu$ the matrices of $(\ad_{Y_k})_{|\n}$ have the block-diagonal form, with the same diagonal blocks as that for the bracket $\mu$.

It follows that the Lie algebra $\bar\g$ defined on the linear space of $\g$ by the bracket $\nu$ is solvable, with the abelian nilradical $\n$ and with an abelian linear complement $\Span(Y_1, \dots, Y_m)$ to $\n$ (note that the nilradical does not increase, as the complex eigenvalues of $(\ad_{Y_k})_{|\n}$ do not change when we pass from $\mu$ to $\nu$).

To finish the proof, by Proposition~\ref{p:clo}, it suffices to construct an inner product on $\bar\g$ whose Ricci curvature is negative, provided that for some $Y \in \Span(Y_1, \dots, Y_m)$ all the eigenvalues of $(\ad_Y)_{|\n}$ (relative to $\nu$) have positive real part. Without loss of generality, suppose that the latter condition is satisfied for $Y=Y_1$ and first choose an inner product $\ip_0$ on $\bar \g$ such that the basis $\cB$ is orthonormal. From the above, for the matrices defined by \eqref{eq:adad} we have $D_i=0, \; B_j=0$, and moreover, the matrices $A_j$ are normal. Then by \eqref{eq:r2}, $R_2=0$ and by \eqref{eq:r1} and \eqref{eq:riccAln}, $R_1= - t A_1^s$, which is negative definite by assumption. Now by Lemma~\ref{l:nonskew} applied to the abelian subalgebra $\Span(Y_1, \dots, Y_m) \subset \gl(\n)$, we can slightly perturb the inner product $\ip_0$ on $\n$ only in such a way that no nontrivial linear combination of the matrices $A_j$ is skew-symmetric. Then by \eqref{eq:r3}, the matrix $R_3$ for the resulting inner product is negative definite. Moreover, $R_2$ is still zero, as we only change the inner product on $\n$, and $R_1$ remains negative definite, if the perturbation is small enough.
\end{proof}

\section{Heisenberg nilradical: proof of Theorem~3}
\label{s:heis}

Let $\g$ be a solvable Lie algebra whose nilradical $\n$ is the Heisenberg Lie algebra of dimension $l=2p+1, \; p \ge 1$ (see eg \cite{RW} for more details of
such algebras). Let $\mathcal{B}=\{X_1, \dots, X_{2p}, Z\}$ be a basis for $\n$ such that $[X_i, X_{p+i}]=Z$ for $i=1, \dots, p$
(and $[X_i,X_j]=0$ if $|i-j| \ne p$) and let $Z$ span the centre $\z$ of $\n$. Then for every $Y \in \g \setminus \n$ there is a well-defined number
$\la(Y)$ such that $[Y,Z]=\la(Y)Z$ and a well-defined linear map $\Phi(Y) := \pi (\ad_Y)_{|\n} \pi^{-1} \in \End(\n/\z)$, where $\pi:\n \to \n/\z$ is the
natural projection (formally one may think that $\pi^{-1}(X)$ is {\it any} $\widetilde{X}\in \n$ such that $\pi(\widetilde{X})=X$ in the definition of $\Phi$).
Let $d_i(Y) \in \bc$, $i=1, \dots, 2p$, be the eigenvalues of $\Phi(Y)$, each listed with its algebraic multiplicity.
The following theorem is stated in the Introduction.

\begin{reptheorem}{t:heis}
A solvable Lie algebra $\g$ with the Heisenberg nilradical admits an inner product of negative Ricci curvature if and only if there exists $Y_+ \in \g$ such that in the above notation,
\begin{equation}\label{eq:heissuf}
    \la(Y_+) + \sum\nolimits_{i: \Re \, d_i(Y_+) < 0} \Re \, d_i(Y_+) > 0.
\end{equation}
\end{reptheorem}

\begin{remark}\label{rem:derheis}
For every element $Y \in \g \setminus \n$, the matrix of the restriction $A_Y$ of $\ad_Y$ to $\n$ relative to $\cB$ has the form $A_Y=\begin{pmatrix} N_Y & 0 \\ v_Y^t & \la(Y) \end{pmatrix}$, where $\la$ is a one-form on $\g$ which vanishes on $\n$, and $JN_Y+N_Y^tJ=\la(Y)J$, where the $2p \times 2p$ matrix $J$ is defined by $J=\begin{pmatrix} 0 & I_p \\ -I_p & 0 \end{pmatrix}$. Then $J^2=-I_{2p}$ and it easily follows that $\Tr N_Y=p \la(Y)$. Moreover, the matrices $N_Y, \; Y \in \g \setminus \n$, commute by the Jacobi identity, as for any $Y,Y' \in \g \setminus \n$ we have $[Y,Y'] \in \n$, so $[A_Y,A_{Y'}]=\begin{pmatrix} 0_{2p \times 2p} & 0 \\ * & 0 \end{pmatrix}$.

Note that $N_Y$ is the matrix of $\Phi(Y)$ relative to the basis $\{X_1, \dots, X_{2p}\}$.
\end{remark}

\begin{proof}
\textbf{Sufficiency. }
Choose a basis $Y_j, \; j=1, \dots ,m$, which complements $\cB$ to a basis for $\g$ in such a way that $Y_1=Y_+$ (this is possible as $\la(Y_+) > 0$ by \eqref{eq:heissuf}, so $Y_+ \notin \n$) and $\la(Y_2)= \ldots = \la(Y_m)=0$. Adding to $Y_j$'s appropriate linear combinations of the $X_i$'s we can always assume that $v_{Y_j}=0$ (note that this does not violate condition \eqref{eq:heissuf} for $Y_1$). Then we have $[A_{Y_j}, A_{Y_k}]=0$ (by Remark~\ref{rem:derheis}), so $[Y_j, Y_k]=\a_{jk}Z$, for all $j,k=1, 2, \dots, m$. But then the Jacobi identity on the triples $(Y_1,Y_j,Y_k), \; 1 < j < k$, implies $\a_{jk}=0$, so we get $[Y_1, Y_j]=a_j Z$ and $[Y_j, Y_k]=0$ for $j,k \ge 2$.

To construct an inner product of negative Ricci curvature on $\g$ we use Proposition~\ref{p:clo}. On the first step, we eliminate the $a_j$'s as follows. For $s \in \br$, define $T_s \in \GL(\g)$ by $T_sY_j=Y_j, \; T_sX_i=e^sX_i, \; T_sZ=e^{2s}Z$ for $j=1, \dots, m, \; i = 1, \dots, 2p$. If $\mu$ is the Lie bracket of $\g$ and $\nu=\lim_{s \to \infty} T_s.\mu$, then for the Lie algebra $\bar \g$ of $\nu$ relative to the basis $\cB$ we have
\begin{equation}\label{eq:heislimit}
    [Y_j,Y_k]=0, \quad (\ad_{Y_j})_{|\n}=\begin{pmatrix} N_j & 0 \\ 0 & \la_j \end{pmatrix}, \quad [X_i, X_r]=j_{ir}Z,
\end{equation}
for $i,r = 1, \dots, 2p, \; j,k=1, \dots, m$, where the matrices $N_j$ commute and satisfy the equation $JN_j+N_j^tJ=\la_jJ$
(and in particular, $\Tr N_j = p \la_j$),  $J=\bigl(j_{ir}\bigr)_1^{2p}$. From the latter equation it follows that the matrices $M_j=N_j-\frac12 \la_j I$ span an $m$-dimensional abelian subalgebra in $\mathfrak{sp}(2p, \br)$. Denote $S_j$ the semisimple part of $M_j$. Note that the matrices $S_j$ commute, are linearly independent (otherwise the nilradical of $\g$ would be bigger) and the maps $\begin{pmatrix} S_j + \frac12 \la_j I& 0 \\ 0 & \la_j \end{pmatrix}$
(relative to $\cB$) are still derivations of $\n$ (see e.~g. \cite[\S 14]{Che}). From \cite[Proposition~11.14]{Ric} it follows that the closure of the orbit of $\Span(M_j)$ under the simultaneous adjoint action of the group $\mathrm{Sp}(2p, \br)$ contains the abelian subalgebra $\Span(S_j)$. Let $h_N \in \mathrm{Sp}(2p, \br), \; N \in \mathbb{N}$, be a sequence of matrices such that $\lim_{N \to \infty} h_N^{-1}M_jh_N=S_j$, for all $j=1, \dots, m$, and define the operators $H_N$ by $H_NY_j=Y_j,\; H_NZ=Z, \; H_N X_i = h_NX_i$. As the group $\mathrm{Sp}(2p, \br)$ acts by automorphisms on the Heisenberg Lie algebra $\n$ we obtain that the Lie bracket $\rho=\lim_{N \to \infty}H_N.\nu$ is given by \eqref{eq:heislimit}, with the $N_j$'s being replaced by their semisimple parts $S_j + \frac12 \la_j I$. By Proposition~\ref{p:clo}, it is sufficient to find an inner product for which the Ricci curvature of the Lie algebra with the bracket $\rho$ is negative. To simplify the notation, in the remaining part of the proof we will keep the notation $\mu$ (instead of $\rho$) for this new Lie bracket and $\g$ for the corresponding Lie algebra.

Recall that the vectors $Y_j$ were chosen in such a way that $\la_1 > 0, \; \la_2= \dots = \la_m = 0$. The following lemma shows that it suffices to construct an inner product of negative Ricci curvature for the one-dimensional extension of $\n$ by $Y_1$.
\begin{lemma}\label{l:heissufreduction}
Let $\ip$ be an inner product on the subalgebra $\g_1=\br Y_1 \oplus \n$ such that $Y_1 \perp \n$ and $Z \perp \m:=\Span(X_1, \dots, X_{2p})$. In the notation of Section~\ref{ss:ric} choose an orthonormal basis $\{e_1, \dots, e_{2p},e_l,f_1\}$ for $(\g_1,\ip)$ in such a way that $f_1 \parallel Y_1$ and $e_l \parallel Z$. Suppose that the matrix $R_1$ defined by \eqref{eq:r1} for $(\g_1,\ip)$ relative to the basis $\{e_1, \dots, e_{2p},e_l,f_1\}$ is negative definite. Then $\g$ admits an inner product of negative Ricci curvature.
\end{lemma}
\begin{proof}
Extend the inner product $\ip$ on $\g_1$ to the inner product $\ip_\ve$ on $\g$ defined by the ortonormal basis $\{e_1, \dots, e_{2p},e_l,f_1, f_2=\ve Y_2, \dots, f_m=\ve Y_m\}$, where $\ve > 0$ (so that, in particular, the restriction of $\ip_\ve$ to $\g_1$ coincides with $\ip$). As $[Y_i, Y_j]=0$, we have $B_j=0$ in the notation of \eqref{eq:adad}. Moreover, $D_l=0$ and for $i < l$, the matrix $D_i$ defined by \eqref{eq:adad} may only have nonzero entries in its bottom row, while the bottom row of $C_i$ is zero, as $[Y_j, \m] \subset \m$. It follows that $R_2=0$. Furthermore, as the inner product $\ip_\ve$ coincides with $\ip$ on $\g_1$, it follows from \eqref{eq:r1} and the assumption that the matrix $R_1$ for $(\g,\ip_\ve)$ can be made negative definite by choosing $\ve$ small enough. Finally, by \eqref{eq:r3}, we have $(R_3)_{jk}=-\Tr(A_j^sA_k^s)$. This matrix is negative semidefinite; it has a zero eigenvalue if and only if the matrices $A_j^s$ are linearly dependent, that is, if and only if a nontrivial linear combination of the $A_j$'s is skew-symmetric. By Lemma~\ref{l:nonskew} applied to the restrictions of the $A_j$ to $\m$ (which span an abelian subalgebra of $\gl(\m)$), we can slightly perturb the inner product $\ip_\ve$ on $\m$ only, so that $R_3$ will become negative definite. Note that $R_1$ will remain negative definite, if the perturbation is small enough and $R_2$ will remain zero, since we are only changing the inner product on $\m$.
\end{proof}

To finish the proof it remains to construct an inner product $\ip$ on $\g_1$ such that $Y_1 \perp \n, \; Z \perp \m$ and $R_1$ is negative definite. Note that for $Y_1 \in \g_1$, the eigenvalues of $(\ad_{Y_1})_{|\n}$ are the same as those for $Y_1 \in \g$. From now on we omit the subscript $1$ in $Y_1, N_1, S_1$ and $\la_1$. According to \eqref{eq:heislimit} we have $(\ad_Y)_{|\m}=N, \; [Y,Z]=\la Z$, with $\Tr N = p\la$ and the matrix $S=N-\frac12 \la I$ being Hamiltonian and semisimple. By \cite[Section~3]{LM}, there exists a basis $\cB'$ for $\m$ relative to which the matrices $S$ and $J$ simultaneously have a (canonical) block-diagonal form $S=\diag(G_1, \dots, G_r), \; J=\diag(J_{2q_1}, \dots, J_{2q_r})$, where $q_i \in \{1,2\}$,  $J_q=\begin{pmatrix} 0 & I_q \\ -I_q & 0 \end{pmatrix}$ and the diagonal blocks $G_i$ have one of the following forms $G_i^{(1)}, G_i^{(2)}, G_i^{(3)}$:
\begin{equation}\label{eq:canhamilton}
    \begin{array}{l}
       G_i^{(1)}=\begin{pmatrix} 0 & \nu_i \\ -\nu_i & 0 \end{pmatrix} \text{ or } \\
      G_i^{(2)}=\begin{pmatrix} \mu_i & 0 \\ 0 & -\mu_i \end{pmatrix}
    \end{array}
    \text{ if } q_i=1, \qquad G_i^{(3)}=\begin{pmatrix} \mu_i & \nu_i & 0 & 0 \\ -\nu_i & \mu_i & 0 & 0 \\ 0 & 0 & -\mu_i & -\nu_i \\ 0 & 0 & \nu_i & -\mu_i \end{pmatrix} \text{ if } q_i=2,
\end{equation}
where $\mu_i \ge 0,\, \nu_i \ne 0$. Then the matrix $N$ has the same block-diagonal decomposition, with the diagonal blocks $G_i + \frac12 \la I_{2q_i}$. The eigenvalues of $N$ are respectively $\frac12 \la \pm \nu_i \mathrm{i}$ or $\pm \mu_i+\frac12 \la$, for every $2 \times 2$ block $G_i+\frac12 \la I_2$ and $\frac12 \la \pm \mu_i \pm \nu_i \mathrm{i}$, for every $4 \times 4$ block $G_i+\frac12 \la I_2$.

Introduce an inner product $\ip$ on $\g_1$ as follows. In the notation of Section~\ref{ss:ric}, set $f_1=Y$ and $e_m=\xi Z, \; \xi \ne 0$. Furthermore, define the inner product on $\m$ so that its matrix $Q$ relative to the basis $\cB'$ has the block-diagonal decomposition $Q=\diag(Q_1, \dots, Q_r)$ which agrees with that for $S$ and $J$, where $Q_i=a_iI_{2q_i}, \; a_i > 0$ (so $Q$ is diagonal and the basis $\cB'$ is orthogonal). We now define $e_i, \; i \le 2p$, to be proportional to the elements of $\cB'$ with the corresponding coefficients (so that $e_i$ are unit) and compute $R_1$ according to \eqref{eq:r1}. Note that relative to the chosen orthonormal basis the matrix $A$ is normal, its symmetric part is diagonal and $t=\Tr A=(p+1)\la$. A direct computation shows that the matrix $R_1$ is diagonal, with
\begin{equation}\label{eq:R1mm}
(R_1)_{mm}=\frac12 \xi^{-2} \sum\nolimits_{i=1}^r q_ia_i^{-2}-(p+1) \la^2,
\end{equation}
and with the $(2p) \times (2p)$-submatrix $\bar R$ in the top left-hand corner having the block-diagonal decomposition $\bar R=\diag(\bar R_1, \dots, \bar R_r)$ which agrees with that for $S$ and $J$, where $\bar R_i$ is the $(2q_i) \times (2q_i)$ matrix of the forms
\begin{align*}
        \bar R_i^{(1)}&=(-\tfrac12 a_i^{-2}\xi^{-2}-\tfrac12 (p+1)\la^2) I_2, \\
        \bar R_i^{(2)}&=\diag(-\tfrac12 a_i^{-2}\xi^{-2}-(p+1)\la(\tfrac12 \la+\mu_i), -\tfrac12 a_i^{-2}\xi^{-2}-(p+1)\la(\tfrac12 \la-\mu_i)), \\
        \bar R_i^{(3)}&=\diag((-\tfrac12 a_i^{-2}\xi^{-2}-(p+1)\la(\tfrac12 \la+\mu_i)) I_2, (-\tfrac12 a_i^{-2}\xi^{-2}-(p+1)\la(\tfrac12 \la-\mu_i)) I_2),
\end{align*}
for the corresponding blocks in \eqref{eq:canhamilton}. It follows that $\bar R_i^{(1)}$ is always negative definite.
Moreover, if $\mu_i \le \frac12 \la$, then $\bar R_i^{(2)}$ and $\bar R_i^{(3)}$ are also negative definite (recall that $\mu_i \ge 0$).
If $\mu_i > \frac12 \la$, then for $\bar R_i^{(2)}$ (respectively, $\bar R_i^{(3)}$) to be negative definite, we have to choose $a_i$ so
that $a_i^{-2}\xi^{-2}>-2(p+1)\la(\tfrac12 \la-\mu_i)$. Finally, according to \eqref{eq:R1mm}, for $(R_1)_{mm}$ to be negative, we can choose $a_i^{-2}>0$ to
be arbitrarily small for those $i$'s which correspond to the blocks $\bar R_i^{(1)}$ and to the blocks $\bar R_i^{(2)}, \bar R_i^{(3)}$ with
$\mu_i \le \frac12 \la$. Using the above inequalities for the remaining $i$'s we obtain that for $(R_1)_{mm}<0$ it is sufficient that
$\la + \sum_{i: \mu_i > \frac12 \la} q_i (\tfrac12 \la-\mu_i) > 0$. But every $\tfrac12 \la-\mu_i$ is the real part of an eigenvalue of the restriction of
$\ad_Y$ to $\m$, with the corresponding multiplicity $q_i$, so that the latter inequality is equivalent to \eqref{eq:heissuf}.

\smallskip
\textbf{Necessity. }Let $\g$ be a solvable Lie algebra whose nilradical $\n$ is the Heisenberg Lie algebra of dimension $l=2p+1$. Suppose $\g$ admits an inner product $\ip$ of negative Ricci curvature. Choose an orthonormal basis $\{e_i,f_j\}$ for $\g$ as in Section~\ref{ss:ric}. We will prove that the vector $f_1$ satisfies inequality \eqref{eq:heissuf}. The proof only uses the fact that the matrix $R_1$ defined by \eqref{eq:r1} is negative definite.

Specify the basis $\{e_i\}$ further, so that $e_{2p+1}$ spans the centre of $\n$ and let $[f_1,e_{2p+1}]=\la e_{2p+1}$ (without loss of generality we may assume that $\la > 0$). Note that $t=\Tr \ad_{f_1}=(p+1)\la$ $(> 0)$. Denote $\m=\n \cap e_{2p+1}^\perp$, and introduce a skew-symmetric operator $K \in \End(\m)$ by $[X_1,X_2]=\<KX_1,X_2\>e_{2p+1}$ for $X_1,X_2 \in \m$.

Relative to the basis $e_i$ for $\n$ the matrices $A_j$ has the form $A_j=\begin{pmatrix} N_j & 0 \\ v_j^t & \la_j \end{pmatrix}$, where $v_j \in \br^{2p}$,
$\la_1=\la$ and $\la_j=0$ for $j > 1$ (Remark~\ref{rem:derheis}). For every eigenvalue $d \in \bc$ of the operator $N_1$ acting on $\m^\bc$ let
$V_d \subset \m^\bc$ be its root subspace, so that $V_d=\cup_{k=1}^\infty \Ker (N_1-d \, \id_\m)^k$. Clearly $V_{\bar d} = \overline{V_d}$ and $\m^\bc=\oplus_{d} V_d$. Let $V_-=\oplus_{d: \Re(d)<0} V_d$ and denote $\m_-=V_- \cap \m$ and $m_-=\dim \m_-$. Note that $m_-$ is the number of the eigenvalues of $N$ with negative real part, counted with their algebraic multiplicity. Denote $\pi_-$ the orthogonal projection to $\m_-$ (both from $\m$ and from $\n$ -- the meaning will be clear from the context).

We want to compute $\Tr (R_1 \pi_-)$. To estimate it we use the following Lemma.
{
\begin{lemma}\label{l:m-}
{\ }
\begin{enumerate}[{\rm (a)}]
  \item \label{it:m-1}
  $[\m_-,\m_-]=0$, so that $\<KX_1, X_2\>=0$ for $X_1,X_2 \in \m_-$.

  \item \label{it:m-2}
  $\Tr (K^tK \pi_-) \le \frac12 \Tr (K^tK)$.

  \item \label{it:m-3}
  The subspace $\m_-$ is an invariant subspace of all the $N_j$.

  \item \label{it:m-4}
  $\Tr ([A_j,A_j^t] \pi_-) \ge -\|\pi_-v_j\|^2$, for all $j=1, \dots, m$.

  \item \label{it:m-5}
  $\Tr (A_1^s \pi_-) =\sum_{d: \Re d<0} d =\sum_{d: \Re d <0} \Re d$, where the sum is taken by all the eigenvalues $d$ of $N_1$, counting their algebraic multiplicity.
\end{enumerate}
\end{lemma}
\begin{proof}
(a) The operator $A_1=\begin{pmatrix} N_1 & 0 \\ v_1^t & \la \end{pmatrix}$ is a derivation of $\n$. Adding an appropriate $\ad_X, \; X \in \n$, we can
eliminate $v_1$, so that the operator $C=\begin{pmatrix} N_1 & 0 \\ 0 & \la \end{pmatrix}$ is again a derivation of $\n$. Then $C$ is a derivation of the
complexified algebra $\n^\bc$. Our arguments are similar to \cite[Lemma~III.3.2]{Hel}. Let $X_1 \in V_{d_1}, \; X_2 \in V_{d_2}, \; \Re d_i < 0$.
Then $(N_1-d_1\id)^{k_1}X_1 = (N_1-d_2\id)^{k_2}X_2 =0$, for some $k_1, k_2 \in \mathbb{N}$, so $(C-d_i\id)^{k_i}X_i =0, \; i=1, 2$. As $C$ is a derivation we obtain by induction that $(C-(d_1+d_2)\id)^k[X_1,X_2] = \sum_{i=0}^k \binom{k}{i} [(C-d_1\id)^{i}X_1, (C-d_2\id)^{k-i}X_2]$. The right-hand side vanishes for $k$ large enough, while the left-hand side equals $(\la-(d_1+d_2))^k[X_1,X_2]$, as $[X_1,X_2]$ is a multiple of $e_{2p+1}$. But $\la > 0$ and $\Re d_1, \Re d_2 < 0$, so $[X_1,X_2]=0$. It follows that $[V_-,V_-]=0$, hence $[\m_-,\m_-]=0$.

Then for $X_1,X_2 \in \m_-$ we have $\<KX_1, X_2\>=\<[X_1, X_2],e_{2p+1}\>=0$.

(b) Let $\{e_i\}_{i=1}^{2p}$ be an orthonormal basis for $\m$ such that $\{e_i\}_{i=1}^{m_-}$ is an orthonormal basis for $\m_-$. We have $\<Ke_i, e_j\>=0$ for $i,j \le m_-$ by \eqref{it:m-1}. Then
\begin{gather*}
\Tr (K^tK)=\sum\nolimits_{i,j=1}^{2p} \<Ke_i,e_j\>^2= \sum\nolimits_{i,j=m_-+1}^{2p} \<Ke_i,e_j\>^2+2\sum\nolimits_{i \le m_- < j} \<Ke_i,e_j\>^2, \\
\Tr (K^tK \pi_-)=\sum\nolimits_{i=1}^{m_-} \<K^tKe_i,e_i\>= \sum\nolimits_{i=1}^{m_-}\sum\nolimits_{j=1}^{2p} \<Ke_i,e_j\>^2=\sum\nolimits_{i \le m_- < j} \<Ke_i,e_j\>^2,
\end{gather*}
and the claim follows.

(c) As $N_j$ commutes with $N_1$, for any $j=1, \dots, m$ (Remark~\ref{rem:derheis}), every root subspace $V_d \subset \m^\bc$ of $N_1$ is $N_j$-invariant. Then $V_- \subset \m^\bc$ is also $N_j$-invariant, as is $\m_-$.

(d) Choose an orthonormal basis $\{e_i\}_{i=1}^{2p}$ for $\m$ as in the proof of \eqref{it:m-2} above and extend it to the basis for $\n$ by the vector $e_{2p+1}$. We have
\begin{align*}
\Tr ([A_j,A_j^t] \pi_-) &=\sum\nolimits_{i=1}^{m_-} \<[A_j,A_j^t]e_i,e_i\>=\sum\nolimits_{i=1}^{m_-} (\|A_j^t e_i\|^2-\|A_j e_i\|^2) \\
&= \sum\nolimits_{i=1}^{m_-}\sum\nolimits_{s=1}^{2p+1} \<A_je_s, e_i\>^2- \sum\nolimits_{i=1}^{m_-}\sum\nolimits_{s=1}^{2p+1} \<A_j e_i,e_s\>^2.
\end{align*}
But the first sum equals $\sum\nolimits_{i=1}^{m_-}\sum\nolimits_{s=1}^{2p} \<A_je_s, e_i\>^2$ (as $e_{2p+1}$ is an eigenvector of $A_j$) and the second one, $\sum\nolimits_{i,s=1}^{m_-} \<A_j e_i,e_s\>^2+\sum\nolimits_{i=1}^{m_-} \<A_j e_i,e_{2p+1}\>^2= \sum\nolimits_{i,s=1}^{m_-} \<A_j e_s,e_i\>^2+ \sum\nolimits_{i=1}^{m_-} \<v_j, e_i\>^2$, by \eqref{it:m-3}, so the claim follows.

(e) With the same choice of basis as in the proof of \eqref{it:m-4} above we have $\Tr (A_1^s \pi_-) = \sum\nolimits_{i=1}^{m_-} \<A_1 e_i,e_i\> = \sum\nolimits_{i=1}^{m_-} \<N_1 e_i,e_i\> = \Tr ({N_1}_{|\m_-})$. Extending $N_1$ to $\m^\bc$ we get $\Tr ({N_1}_{|\m_-})=\Tr ({N_1}_{|V_-})=\sum_{d: \Re d<0} \Tr ({N_1}_{|V_d})= \sum_{d: \Re d<0} d \dim V_d$, where the sum is taken by all the eigenvalues $d$ of $N_1$ \emph{without} counting the multiplicity.
\end{proof}
}

Choosing an orthonormal basis $\{e_i\}$ for $\n$ as in the proof of Lemma~\ref{l:m-} (so that $e_1, \dots e_{m_-}$ is a basis for $\m_-$) we obtain by \eqref{eq:riccAln} $\Tr (\Ric^\mathfrak{n} \pi_-)=-\frac{1}{2}\sum\nolimits_{s=1}^{2p+1}\sum\nolimits_{i=1}^{m_-}  \|[e_s,e_i]\|^2 = -\frac{1}{2}\sum\nolimits_{i=1}^{m_-} \|Ke_i\|^2 = -\frac{1}{2} \Tr (K^tK \pi_-)$.

Then using Lemma~\ref{l:m-} we get from \eqref{eq:r1}
\begin{equation} \label{eq:Trpi-}
    \Tr (R_1 \pi_-) \ge -\frac{1}{4} \Tr (K^tK) - \frac{1}{2} \sum\nolimits_{j=1}^{m}\|\pi_-v_j\|^2 - t \sum\nolimits_{d: \Re d <0} \Re d,
\end{equation}
where the sum is taken by all the eigenvalues $d$ of $N_1$, counting their algebraic multiplicity.

On the other hand, from \eqref{eq:r1} and \eqref{eq:riccAln} (and using the fact that $e_{2p+1}$ is a common eigenvector of all the $A_j$) we have
\begin{align*}
\<R_1 e_{2p+1},e_{2p+1}\> &\ge \frac{1}{4} \sum\nolimits_{i,s=1}^{2p} \<[e_i,e_s],e_{2p+1}\>^2 + \frac{1}{2} \sum\nolimits_{j=1}^{m} (\|A_j^te_{2p+1}\|^2-\|A_j e_{2p+1}\|^2) - t \la \\
&= \frac{1}{4} \sum\nolimits_{i,s=1}^{2p} \<Ke_i,e_s\>^2 + \frac{1}{2} \sum\nolimits_{j=1}^{m} \sum\nolimits_{i=1}^{2p} \<A_j^te_{2p+1},e_i\>^2 - t \la \\
&= \frac{1}{4} \Tr (K^tK) + \frac{1}{2} \sum\nolimits_{j=1}^{m} \|v_j\|^2 - t \la.
\end{align*}
Adding this to \eqref{eq:Trpi-} and using the fact that $\pi_-$ is positive semidefinite and $R_1$ is negative definite we get $\la+\sum\nolimits_{d: \Re d <0} \Re d > 0$, which is equivalent to \eqref{eq:heissuf}.
\end{proof}

\section{Filiform nilradical: proof of Theorem~4}
\label{s:fili}

A nilpotent Lie algebra of dimension $l$ is called \emph{filiform} if it has the maximal possible degree of nilpotency ($\Leftrightarrow$ the longest possible lower central series). Filiform Lie algebras have been introduced in \cite{Ver} and have been given a great deal of attention thereafter.

In this section, we consider solvable Lie algebras $\g$ whose nilradical is the so called \emph{standard filiform Lie algebra}. The latter is defined as the $l$-dimensional Lie algebra $L_l$ having a basis $X_1, \dots, X_l$ such that $[X_1, X_i]=X_{i+1}, \; i=2, \dots, l-1, \; [X_1, X_l]=0$, and $[X_i,X_j]=0$ when $i,j \ge 2$. On the question of how restrictive the assumption of standardness is, note that any filiform Lie algebra admits a basis for which the former relations are satisfied (but in general, not the latter ones), and that any filiform algebra of dimension $l$ degenerates to $L_l$. Note also that a ``typical" filiform Lie algebra of dimension $l \ge 8$ is characteristically nilpotent, hence cannot be the nilradical of anything except for itself (any solvable extension of it is nilpotent).

Let $\g$ be a solvable Lie algebra with the nilradical $L_l$. We can assume that $l \ge 4$ (as $L_2$ is abelian and $L_3$ is the Heisenberg algebra). The algebra $L_l$ has a (unique) codimension one abelian ideal $\ig=\Span(X_2, \dots, X_l)$ and the one-dimensional centre $\br X_l$. Both of them are characteristic ideals of $L_l$ (they are invariant under the action of any derivation on $L_l$; see Remark~\ref{rem:fili} below). Let $\la$ and $\iota$ be one-forms on $\g$ defined as follows: for $Y \in \g, \; [Y,X_l]=\la(Y) X_l$ and $\iota(Y)=\Tr((\ad_Y)_{|\ig})$. The following theorem is stated in the Introduction.

\begin{reptheorem}{t:fili}
Let $\g$ be a solvable Lie algebra with the nilradical $\n=L_l, \; l \ge 4$. The algebra $\g$ admits an inner product of negative Ricci curvature if and only if there exists $Y \in \g$ such that $\la(Y) > 0$ and $\iota(Y) > 0$.
\end{reptheorem}

\begin{remark}\label{rem:fili}
It is easy to see (and is well-known) that relative to the basis $\cB=\{X_i\}$ for $L_l$, any derivation of $L_l$ has the lower-triangular form
\begin{equation}\label{eq:derfili}
    \left(
\begin{array}{c@{}c@{}}
 \begin{array}{ccc}
         a & & \\
         & d & \\
        & & a+d \\
  \end{array} & \mbox{\Huge $0$ } \\
  \mbox{\Huge $\ast$ } & \begin{array}{ccc}
                       2a+d & & \\ & \ddots & \\ & & (l-2)a+d \\
                      \end{array}
\end{array}\right),
\end{equation}
where $a,d \in \br$ (with some additional relations on the entries below the diagonal). It follows that any solvable extension of $L_l$ of rank $m \ge 3$ contains a nilpotent derivation $\ad_Y, \; Y \notin L_l$, so that its nilradical is bigger than $L_l$. Therefore $m=1$ or $m=2$. But if $m=2$, there exists $Y \in \g$ such that $(\ad_Y)_{|\n}$ has the form \eqref{eq:derfili}, with $a,d > 0$. Then the inequalities $\la(Y) > 0, \; \iota(Y) > 0$ are trivially satisfied for such $Y$ and the existence of an inner product of negative Ricci curvature follows from Theorem~\ref{t:neg}\eqref{it:neg2}.

We can therefore assume that $m=1$. Then for an arbitrary $Y \notin \n$, the matrix $(\ad_Y)_{|\n}$ has the form \eqref{eq:derfili}, with $a,d$ not simultaneously zero, and the inequalities $\la(Y) > 0$ and $\iota(Y) > 0$ are equivalent to
\begin{equation}\label{eq:filiad}
    (l-2)a+d > 0, \qquad (l-2)a+2d > 0,
\end{equation}
respectively.
\end{remark}

\begin{proof}
By Remark~\ref{rem:fili} we can assume that $\g$ is a one-dimensional extension of $L_l$. Choose and fix a vector $Y \in \g \setminus \n$. Then $(\ad_Y)_{|\n}$ has the form \eqref{eq:derfili}, and $\Tr \ad_Y= (l-1)d+(\frac12 (l-1)(l-2)+1)a=\frac{2}{l-2}((l-2)a+d)+\frac{l(l-3)}{2(l-2)}((l-2)a+2d)$. If $\Tr \ad_Y=0$, then \eqref{eq:filiad} is violated and no inner product with negative Ricci curvature exists, as $\g$ is unimodular (Theorem~\ref{t:dm}). We can therefore assume that $t = \Tr \ad_Y > 0$.

\smallskip
\textbf{Sufficiency.} Suppose the inequalities \eqref{eq:filiad} are satisfied. Let $N$ be a positive derivation of $\n$ which is diagonal relative to $\cB$, say $N = \diag(1,2, \dots,l)$. For $s > 0$ define $T_s \in \End(\g)$ by $T_sY=Y$ and $T_sX=e^{sN}X$, for $X \in \n$. When $s \to \infty$, the Lie algebra $\g$ degenerates to the Lie algebra $\bar g$ with the same nilradical $\n=L_l$ and with $(\ad_Y)_{|\n} = \diag(a,d,a+d,2a+d, \dots, (l-2)a+d)$ (relative to $\cB$). By Proposition~\ref{p:clo} it suffices to construct an inner product of negative Ricci curvature on $\bar \g$. In the notation of Section~\ref{ss:ric} take $f_1=Y$ and $e_i=a_iX_i, \; a_i > 0$, for $i=1, \dots, l$. An easy calculation shows that $R_2=0$ and $R_3 < 0$. Furthermore, from \eqref{eq:r1} we obtain that the matrix $R_1$ is diagonal, with the entries
\begin{gather*}
    (R_1)_{11}=-\sum\nolimits_{i+2}^{l-1}\xi_i-ta,\quad (R_1)_{22}=-\xi_2-td, \quad (R_1)_{ll}=\xi_{l-1}-t((l-2)a+d)), \\ (R_1)_{ii}=\xi_{i-1}-\xi_i-t((i-1)a+d), \text{ for } i=3,\dots, l-1,
\end{gather*}
where $\xi_i=\frac12 a_1^2a_i^2a_{i+1}^{-2}, \; i=2, \dots, l-1$, and $t=\Tr A= (l-1)d+(\frac12 (l-1)(l-2)+1)a > 0$. In a Euclidean space $\br^l$, with an orthonormal basis $E_i$, introduce the vectors $F_1=-E_1-E_2+E_3, \, F_2=-E_1-E_3+E_4, \dots, F_{l-2}=-E_1-E_{l-1}+E_{l}$ and the vectors $V_1=E_1+\sum_{i=3}^l (i-2)E_i, \; V_2=\sum_{i=2}^l E_i$. Then the fact that $R_1$ is negative definite is equivalent to the fact that all the components of the vector $\sum_{i=1}^{l-2}\xi_iF_i-taV_1-tdV_2$ are negative, that is, to the fact that the vector $aV_1+dV_2-t^{-1}\sum_{i=1}^{l-2}\xi_iF_i$ belongs to the first octant of $\br^n$. Note that given any $\xi_i > 0$ we can easily find the corresponding $a_i$'s by taking $a_1, a_2 > 0$ arbitrarily and defining $a_{i+1}=a_1a_i(2\xi_i)^{1/2}$ for $i=2, \dots, l-1$, so for the existence of an inner product
such that $R_1$ is negative definite it is sufficient that $aV_1+dV_2$ belongs to the open convex hull of the vectors $F_1, F_2, \dots, F_{l-2},E_1,E_2, \dots, E_l$. But $E_l=F_{l-2}+E_{l-1}+E_{l-2}$, and then $E_{l-1}=F_{l-3}+E_{l-2}+E_{l-3}$, and so on, up to $E_3=F_1+E_2+E_1$, so that the convex hull of the vectors $F_1, F_2, \dots, F_{l-2},E_1,E_2, \dots, E_l$ is the same as the convex hull of the vectors $F_1, F_2, \dots, F_{l-2},E_1,E_2$. These vectors form a basis for $\br^l$. Expanding $V_1$ and $V_2$ by that basis we get
\begin{gather*}
    V_1=\sum\nolimits_{i=1}^{l-2}\Big(\sum\nolimits_{j=i}^{l-2}j \Big)F_i+(\tfrac16(l-2)(l-1)(2l-3)+1)E_1+\tfrac12(l-2)(l-1)E_2, \\
    V_2=\sum\nolimits_{i=1}^{l-2}(l-1-i)F_i+\tfrac12(l-2)(l-1)E_1+(l-1)E_2,
\end{gather*}
so the vector $aV_1+dV_2$ lies in the open convex hull of the vectors $F_1, F_2, \dots, F_{l-2},E_1,E_2$ if and only if the following inequalities are satisfied:
\begin{gather*}
    \Big(\sum\nolimits_{j=i}^{l-2}j\Big) \, a+(l-1-i)d > 0, \quad i=1, \dots, l-2, \\
    (\tfrac16(l-2)(l-1)(2l-3)+1) a + \tfrac12(l-2)(l-1) d > 0, \quad \tfrac12(l-2)(l-1) a + (l-1)d > 0.
\end{gather*}
These inequalities are equivalent to the inequalities $\kappa_i a + d > 0$, where $\kappa_i > 0$ are defined by $\kappa_i = (\sum\nolimits_{j=i}^{l-2}j)/(l-1-i)$ for $i=1, \dots, l-2$, and $\kappa_{l-1}=(\tfrac16(l-2)(l-1)(2l-3)+1)(\tfrac12(l-2)(l-1))^{-1}, \; \kappa_l=\tfrac12(l-2)$, and hence are equivalent to the two inequalities $\kappa_{\max} a + d > 0, \; \kappa_{\min} a + d > 0$, where $\kappa_{\max} = \max \kappa_i, \; \kappa_{\min} = \min \kappa_i$. As $\kappa_i$ increase when  $i=1, \dots, l-2$ and $\kappa_l < \kappa_1$, and moreover, $\kappa_l < \kappa_{l-1} < \kappa_{l-2}$, we get $\kappa_{\max} = \kappa_{l-2}=l-2, \; \kappa_{\min} = \kappa_l=\tfrac12(l-2)$. This gives $(l-2)a+d > 0, \; (l-2)a+2d > 0$, which is satisfied by the assumption (see \eqref{eq:filiad}, Remark~\ref{rem:fili}).

\smallskip
\textbf{Necessity.} Suppose $\ip$ is an inner product of negative Ricci curvature on $\g$. Take $f_1$ to be the unit vector orthogonal to $\n$ with $\Tr \ad_{f_1} > 0$. The inequality $\la(f_1) > 0$ follows from Theorem~\ref{t:neg}\eqref{it:neg1}.
To prove the inequality $\iota(f_1) > 0$ introduce an orthonormal basis $e_i$ for $\n$ in such a way that $e_1 \perp \ig$. Note that $[e_i,e_j]=0$ for $i,j > 1$, as $\ig$ is abelian, and that $[e_1,\ig] \subset \ig$, as $\ig$ is an ideal. Then by \eqref{eq:riccAln} we obtain
\begin{align*}
\sum\nolimits_{j=2}^{l}\<\Ric^\mathfrak{n}e_j,e_j\> &= -\frac{1}{2} \sum\nolimits_{j=2}^{l}\sum\nolimits_{i=1}^{l} \|[e_i,e_j]\|^2 +\frac{1}{4} \sum\nolimits_{j=2}^{l}\sum\nolimits_{i,k=1}^{l}\<[e_i,e_k],e_j\>^2\\
&=-\frac{1}{2} \sum\nolimits_{j=2}^{l} \|[e_1,e_j]\|^2 +\frac{1}{2} \sum\nolimits_{j,k=2}^{l}\<[e_1,e_k],e_j\>^2\\
&=0.
\end{align*}
It now follows from \eqref{eq:r1} that
\begin{align*}
\sum\nolimits_{j=2}^{l}\<R_1e_j,e_j\> &= \sum\nolimits_{j=2}^{l}\<(\tfrac{1}{2} [A_1,A_1^t]- t A_1^s)e_j,e_j\> \\
&= \Tr(\tfrac{1}{2} [A_1,A_1^t]- t A_1^s)-\<(\tfrac{1}{2} [A_1,A_1^t]- t A_1^s)e_1,e_1\>\\
&= - t (\Tr A_1- \<A_1 e_1,e_1\>)-\tfrac{1}{2} (\|A_1^t e_1\|^2-\|A_1 e_1\|^2).
\end{align*}
The matrix $A_1$ relative to the basis $e_i$ is similar to the matrix \eqref{eq:derfili}, with some $a,d \in \br$ such that $t=\Tr(A_1) > 0$, and as $\ig=\Span(X_2, \dots, X_l)=\Span(e_2, \dots, e_l)$ we have
\begin{equation*}
    A_1=\left(
          \begin{array}{cc}
            a & 0 \dots 0 \\
            v & M \\
          \end{array}
        \right),
\end{equation*}
where $v \in \br^{l-1}$ and $M$ is an $(l-1) \times (l-1)$-matrix with the eigenvalues $d,a+d,2a+d, \dots, (l-2)a+d$. In particular, $\iota(f_1) = \Tr M=\Tr A_1- \<A_1 e_1,e_1\>$ and $\|A_1^t e_1\|^2-\|A_1 e_1\|^2=-\|v\|^2$. Then from the above
\begin{equation*}
    \sum\nolimits_{j=2}^{l}\<R_1e_j,e_j\> = - t \iota(f_1)+\tfrac{1}{2} \|v\|^2,
\end{equation*}
so $\iota(f_1) > 0$, as required.
\end{proof}

\section{Open questions}
\label{s:q}

In this section, motivated by the above theorems (and by the theory of Einstein solvmanifolds \cite{Heb, LW}), we collect some open questions and conjectures.

\begin{enumerate}
  \item
  Is it true that we can always reduce the rank? In other words, is the following true (or at least, one way true): ``a solvable algebra $\g$ with the nilradical $\n$ admits an inner product with negative Ricci curvature if and only if there exists $Y \notin \n$ such that the subalgebra $\br Y \oplus \n \in \g$ admits such an inner product"?

  \item
  Are there some ``best" solvable metric Lie algebras of negative Ricci curvature? More specifically: can we modify a given solvable Lie algebra with $\Ric < 0$ by degeneration and then choose an inner product in such a way that the resulting algebra is nicer -- e.g., the restrictions $(\ad_Y)_{|\n}, \; Y \in \ag$, are semisimple, the subspace $\ag$ is abelian and $\n$ and $\a$ are invariant subspaces of $\Ric$ (so that in the notation of Section~\ref{ss:ric}, $R_2=0$ and hence the negativity of $\Ric$ only depends on $R_1$)?

  \item
  A stronger form of the first question: is it true that a solvable Lie algebra $\g$ admits an inner product with $\Ric < 0$ if and only if there exists a vector $Y \in \g$ such that the real parts of the eigenvalues of the restriction of $\ad_Y$ to the nilradical $\n$ satisfy certain linear inequalities which are determined by the structure of $\n$? Such inequalities may represent the fact that $\Re \, (\ad_Y)_{|\n}$ belongs to a certain open convex hull. An alternative description may be the following. For a linear operator $A$ denote $A^{\br}$ its real semisimple part. If $A$ is a derivation of a Lie algebra $\n$, then $A^{\br}$ also is (see e.g. \cite[Section~3.3]{Heb}). Relative to some basis for $\n$, $A^{\br}$ is diagonal. Consider the space $\mathfrak{t}$ of all the diagonal matrices relative to that basis (more invariantly, the Cartan subalgebra of the centraliser of $A^{\br}$ in $\gl(\n)$) and call $A$ \emph{semipositive} if $A^{\br}$ belongs to the orthogonal projection of the first octant of $\mathfrak{t}$ (the cone of the matrices with all the diagonal entries positive) to $\mathfrak{t} \cap \Der(\n)$ via the restriction of the Killing form of $\gl(\n)$ to $\mathfrak{t}$. Is it true then that a solvable Lie algebra $\g$ admits an inner product of negative Ricci curvature if and only if there exists a vector $Y \in \g$ such that the restriction of $\ad_Y$ to the nilradical $\n$ is semipositive?

\end{enumerate}

\vspace{10mm}

\end{document}